\documentclass[11pt,reqno]{amsart}

\usepackage[utf8]{inputenc}
\usepackage[T1]{fontenc}
\usepackage[english]{babel}
\usepackage{textcomp}
\usepackage{amsmath,amssymb,amsthm}
\usepackage{lmodern}
\usepackage[a4paper]{geometry}
\usepackage{xcolor, pict2e}
\usepackage{microtype}
\usepackage{caption}
\usepackage{listings}
\usepackage{multicol}
\usepackage{moreverb}
\usepackage{hyperref}
\hypersetup{pdfstartview=XYZ}
\usepackage{wrapfig}
\usepackage[sans]{dsfont}
\usepackage{booktabs}

\usepackage{mathrsfs}
\usepackage{tikz, pgfplots}
\usepgfplotslibrary{groupplots}
\usetikzlibrary{positioning}
\usepackage{amsmath,amssymb}
\usetikzlibrary{decorations.pathreplacing}

\usepackage{ytableau}
\usepackage{array}

\usepackage{mathdots}
\usepackage{hyperref}
\hypersetup{
    colorlinks=true,
    linkcolor=blue,
    citecolor=magenta,
    urlcolor=blue,
    pdfborder={0 0 0}
}

\definecolor{fond}{rgb}{0.05,0.05,0.25}

\DeclareMathOperator{\ch}{\mathrm{ch}}
\DeclareMathOperator{\Id}{\mathrm{Id}}

\DeclareMathOperator{\cyc}{\mathrm{cyc}}

\newcommand{\kS}{{\ensuremath{\mathfrak{S}}} }

\usepackage[font=sf, labelfont={sf,bf}, margin=1cm]{caption}
\usepackage{enumitem}
\setlist[itemize,1]{nosep}
\setlist[enumerate,1]{itemsep=0pt,label=(\alph*)}

\usepackage{float}
\usepackage[caption = false]{subfig}
\usepackage{graphicx}

\newtheorem*{theorem*}{Theorem}
\newtheorem{theorem}{Theorem}[section]
\newtheorem{proposition}[theorem]{Proposition}
\newtheorem{lemma}[theorem]{Lemma}

\usepackage{todonotes}

\newcommand{\cE}{{\ensuremath{\mathcal E}} }

\newcommand{\ag}{\left\{ } 

\newcommand{\ad}{\right\} }

\newcommand{\pg}{\left(} 
\newcommand{\pd}{\right)}
\newcommand{\bg}{\left|}
\newcommand{\bd}{\right|}
\newcommand{\lf}{\left\lfloor}
\newcommand{\rf}{\right\rfloor}
\newcommand{\lc}{\left\lceil}
\newcommand{\rc}{\right\rceil}

\DeclareMathOperator{\ST}{ST}
\DeclareMathOperator{\RT}{RT}
\DeclareMathOperator{\height}{ht}

\makeatletter
\newcommand*\bigcdot{\mathpalette\bigcdot@{.5}}
\newcommand*\bigcdot@[2]{\mathbin{\vcenter{\hbox{\scalebox{#2}{$\m@th#1\bullet$}}}}}
\makeatother

\numberwithin{equation}{section}

\pgfplotsset{compat=1.18}

\begin{document}

\title{River ribbon tableaux and symmetric group characters}
\author{Lucas Teyssier}
\address{Université de Lorraine}
\email{lucas.teyssier@univ-lorraine.fr}
\date{}

\begin{abstract}
    We introduce river ribbon tableaux, where a river of non-trivial ribbons is placed between two banks of single boxes. Compared to other ribbon peelings, this reduces the geometric differences across diagrams of different shapes, and the maximal number of ribbon configurations.

This allows us to prove a uniform character bound for symmetric groups which, specialized to balanced diagrams, proves the conjecture of Moore and Russell. 

As an application, we prove a conjecture of Dousse and Féray on the number of standard tableaux of skew shapes.
\end{abstract}

\maketitle

\section{Introduction}
\subsection{Main results}

Let $\lambda$ be a Young diagram of size $|\lambda| = n$, which we always identify with the associated irreducible representation of $\kS_n$. Let $s = s(\lambda)$ be its maximal hook length, $\delta = \delta(\lambda)$ be its diagonal length, and $d_\lambda$ be its dimension.
Let $\sigma\in \kS_n$ be a permutation. We denote its number of cycles by $\cyc(\sigma)$ and we set $|\sigma| = n-\cyc(\sigma)$, which is the minimal number of transpositions one needs to multiply to obtain $\sigma$. 
We use the notation $\ch^\lambda(\sigma)$ for characters, $\chi^\lambda(\sigma) = \ch^\lambda(\sigma)/d_\lambda$ for renormalized characters, $\widehat{\kS_n}$ for the set of irreducible representations of $\kS_n$, and we may write $\lambda \vdash n$ instead of $\lambda \in \widehat{\kS_n}$. Finally, we say that $\lambda$ is $C$-balanced if $s(\lambda) \leq C \sqrt{|\lambda|}$.

Our main result is the following.

\begin{theorem}\label{thm: main general form}
Let $n\geq 1$, $\lambda\vdash n$, and $\sigma\in \kS_n$. Then
\begin{equation}\label{eq: main general form only with power |sigma|}
     \bg \chi^\lambda(\sigma) \bd \leq \pg 5\frac{\delta s^2}{n^2} \pd^{|\sigma|}.
\end{equation}
\end{theorem}
Under the mild shape assumption $\delta s = O(n)$, the exponentiated term becomes $O\pg \frac{s}{n}\pd$, which does not have the permutation part $\frac{|\sigma|}{n}$ from the term $O\pg \max\pg \frac{s}{n},\frac{|\sigma|}{n}\pd \pd$ of Féray and Śniady \cite{FeraySniady2011}.

Specializing Theorem~\ref{thm: main general form} to balanced diagrams, we obtain the following bound which 
was conjectured by Moore and Russell (see \cite[Conjecture 3 and Section 1.7.1]{RattanSniady2008}).

\begin{theorem}\label{thm: main balanced Moore Russell form}
    Let $C> 1$. Let $n\geq 1$, $\lambda\vdash n$ be $C$-balanced, and $\sigma\in \kS_n$. Then
\begin{equation}\label{eq: Bound Moore Russell form}
        \bg \chi^\lambda(\sigma) \bd \leq  \pg \frac{5C^2}{\sqrt{n}}\pd^{|\sigma|}.
    \end{equation}
\end{theorem}

\subsection{Link with previous results}
Bounds on characters of symmetric groups have a long history. There are two main lines of results.

\medskip

One main line consists of bounds for relatively \textit{thick} diagrams.
For balanced diagrams, \eqref{eq: Bound Moore Russell form} was proved by Biane \cite{Biane1998RepresentationsFreeProbability, Biane2001ApproximateFactorizationConcentration}, building on Kerov's work \cite{Kerov1998InterlacingMeasures}, for permutations with bounded support. It was then extended to permutations with support $O(\sqrt{n})$ by Rattan and Śniady \cite{RattanSniady2008}, together with weaker bounds for permutations with support $\omega(\sqrt{n})$. Then Féray and Śniady \cite{FeraySniady2011} proved a more general bound which for balanced diagrams and permutations with support $\omega(\sqrt{n})$ gives
\begin{equation}
    |\chi^\lambda(\sigma)| \leq \pg O(1) \frac{|\sigma|}{n}\pd^{|\sigma|}.
\end{equation}
This bound becomes weaker as the support of $\sigma$ grows but is better than the previous bound from \cite{RattanSniady2008}. Finally, we recently showed in \cite{Teyssier2025SkewDimensionsAndCharacters} that the bound
\begin{equation}\label{eq: Bound balanced thick hook}
        \bg \chi^\lambda(\sigma) \bd \leq  \pg \frac{O(1)}{\sqrt{|\sigma|}}\pd^{|\sigma|}
    \end{equation}
holds for any support size, which is sharp for permutations with support $\Theta(n)$ but becomes weaker as the support decreases. Our Theorem~\ref{thm: main balanced Moore Russell form} completes this line of results, by proving the conjecture of Moore and Russell.

Our bound also has a numerical advantage compared to the previous bounds \cite{RattanSniady2008, FeraySniady2011, Teyssier2025SkewDimensionsAndCharacters}: the universal constant 5 is very small, and Theorem~\ref{thm: main balanced Moore Russell form} therefore already gives non-trivial bounds for reasonably small diagrams. 

\medskip

A second main line of character bounds is intimately related to generalizing the work of Diaconis and Shahshahani \cite{DiaconisShahshahani1981} on random transpositions to other conjugacy classes. Such bounds, often called exponential character bounds, include \cite{Roichman1996, MullerSchlagePuchta2007precutoff, LarsenShalev2008} and our recent results \cite{TeyssierThevenin2025virtualdegreeswitten, Olesker-TaylorTeyssierThevenin2025sharpboundsandcutoff}, which improve these previous bounds at all scales. These bounds typically rely on the Murnaghan--Nakayama rule, while \cite{RattanSniady2008,FeraySniady2011} do not.

In the last years hypercontractivity arguments have been developed, see \cite{KeevashLifshitz2023sharp, KeevashLifshitzLongMinzer2024HypercontractivityGlobalFunctions}, which also led to some character bounds \cite{LifshitzMarmor2024charactershypercontractivity}. While the character bounds obtained combinatorially are stronger for applications to mixing times, hypercontractivity has also led to striking applications on problems related to exponential character bounds, including a recent proof of the robust Thompson conjecture for alternating groups \cite{KellerLifshitzMarmorSheinfeld2026ThompsonCoveringNumbers}.

\section{River ribbon tableaux}
\subsection{Ribbon tableaux and the Murnaghan--Nakayama rule}
A ribbon is a connected skew Young diagram which does not contain any 2 by 2 square. A ribbon tableau of a Young diagram $\lambda$, of weight $\alpha = (\alpha_1, \ldots, \alpha_m)$, is a tiling of $\lambda$ with ribbons $\rho_1, \ldots, \rho_m$ of respective sizes $\alpha_1, \ldots, \alpha_m$ such that, filling the boxes of each $\rho_i$ with the number $i$, the resulting filling of $\lambda$ is (weakly) increasing along rows and columns. See Figure~\ref{fig: two examples of ribbon tableaux} for an example. We denote the set of such ribbon tableaux by $\RT(\lambda, \alpha)$. The same definition extends to skew diagrams.

\begin{figure}[!ht]
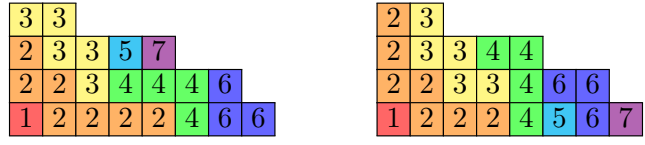

\ytableausetup{boxsize=1.1em}
    \begin{center}
\begin{ytableau}
    \none   &*(yellow!60)3&*(yellow!60)3 \\
    \none   &*(orange!60)2&*(yellow!60)3&*(yellow!60)3&*(cyan!60)5&*(violet!60)7 \\
    \none   &*(orange!60)2&*(orange!60)2&*(yellow!60)3&*(green!60)4&*(green!60)4&*(green!60)4&*(blue!60)6 \\
    \none   &*(red!60)1&*(orange!60)2&*(orange!60)2&*(orange!60)2&*(orange!60)2&*(green!60)4&*(blue!60)6&*(blue!60)6 \\
\end{ytableau}
\quad \quad
\begin{ytableau}
    \none   &*(orange!60)2&*(yellow!60)3 \\
    \none   &*(orange!60)2&*(yellow!60)3&*(yellow!60)3&*(green!60)4&*(green!60)4 \\
    \none   &*(orange!60)2&*(orange!60)2&*(yellow!60)3&*(yellow!60)3&*(green!60)4&*(blue!60)6&*(blue!60)6 \\
    \none   &*(red!60)1&*(orange!60)2&*(orange!60)2&*(orange!60)2&*(green!60)4&*(cyan!60)5&*(blue!60)6&*(violet!60)7 \\
\end{ytableau}
\end{center}
\caption{Two ribbon tableaux of $\lambda = [8,7,5,2]$ with weight $\alpha = (1,7,5,4,1,3,1)$. Colors are redundant.}
\label{fig: two examples of ribbon tableaux}
\end{figure}

One important use of ribbon tableaux is to bound characters. Indeed, the Murnaghan--Nakayama rule (\cite[Theorem 3.10]{LivreMeliot2017RepresentationTheoryofSymmetricGroups}) enables computing characters $\ch^\lambda(\sigma)$ as a signed sum over all ribbon tableaux $\RT(\lambda, \alpha)$, where $\alpha$ represents the cycle type of $\sigma$. 

To be more specific, the height of a ribbon is the coordinate of its highest box minus the coordinate of its lowest box, and the height of a ribbon tableau $T$, which we denote by $\height(T)$, is the sum of the heights of all its ribbons.
The Murnaghan--Nakayama rule then states that for every Young diagram $\lambda$ and any $\sigma \in \kS_{|\lambda|}$, denoting by $\alpha$ \textit{any} ordering  of the cycle lengths of $\sigma$, we have
\begin{equation}
    \ch^\lambda(\sigma) = \sum_{T\in \RT(\lambda, \alpha)} (-1)^{\height(T)}.
\end{equation}
Therefore, by a triangle inequality, we have
\begin{equation}\label{eq: MN with triangle inequality and min}
    \bg \ch^\lambda(\sigma)\bd \leq \min_\alpha \bg \RT(\lambda, \alpha)\bd,
\end{equation}
where the minimum is over all possible orderings of the cycles of $\sigma$.

\medskip

A possible way to bound characters is therefore to find a \textit{good} $\alpha$, that is an $\alpha$ for which the ribbon tableaux have a geometric structure that may be analysed, and for which $\bg \RT(\lambda, \alpha)\bd$ is as small as possible.

\medskip

We refer to the companion paper \cite[Section 3.1]{Teyssier2025SkewDimensionsAndCharacters} for a detailed exposition on the Murnaghan--Nakayama rule with several additional figures.

\subsection{What is problematic with typical orderings and tableaux shapes}
Bounding characters via \eqref{eq: MN with triangle inequality and min} has been done many times in the past. Depending on the shapes $\lambda$ and permutations $\sigma$, this led to some very good bounds and some weaker ones.

\medskip

A first possible choice of $\alpha$ is to give a low index to larger ribbons. Schematically, this gives ribbon tableaux as in Figure~\ref{fig: corner ribbon peeling}.

\begin{figure}[!ht]
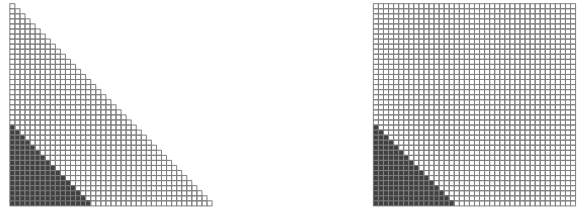

    \centering
{
\ytableausetup{boxsize=0.16em, boxframe = 0.15pt}
\color{black!50}
\begin{ytableau}
    \none & \\
    \none && \\
    \none &&& \\
    \none &&&& \\
    \none &&&&& \\
    \none &&&&&& \\
    \none &&&&&&& \\
    \none &&&&&&&& \\
    \none &&&&&&&&& \\
    \none &&&&&&&&&& \\
    \none &&&&&&&&&&& \\
    \none &&&&&&&&&&&& \\
    \none &&&&&&&&&&&&& \\
    \none &&&&&&&&&&&&&& \\
    \none &&&&&&&&&&&&&&& \\
    \none &&&&&&&&&&&&&&&& \\
    \none &&&&&&&&&&&&&&&&& \\
    \none &&&&&&&&&&&&&&&&&& \\
    \none &&&&&&&&&&&&&&&&&&& \\
    \none &&&&&&&&&&&&&&&&&&&& \\
    \none &&&&&&&&&&&&&&&&&&&&& \\
    \none &&&&&&&&&&&&&&&&&&&&&& \\
    \none &&&&&&&&&&&&&&&&&&&&&&& \\
    \none &&&&&&&&&&&&&&&&&&&&&&&& \\
    \none &*(darkgray)&&&&&&&&&&&&&&&&&&&&&&&& \\
    \none &*(darkgray)&*(darkgray)&&&&&&&&&&&&&&&&&&&&&&&& \\
    \none &*(darkgray)&*(darkgray)&*(darkgray)&&&&&&&&&&&&&&&&&&&&&&&& \\
    \none &*(darkgray)&*(darkgray)&*(darkgray)&*(darkgray)&&&&&&&&&&&&&&&&&&&&&&&& \\
    \none&*(darkgray)&*(darkgray)&*(darkgray)&*(darkgray)&*(darkgray)&&&&&&&&&&&&&&&&&&&&&&&& \\
    \none &*(darkgray)&*(darkgray)&*(darkgray)&*(darkgray)&*(darkgray)&*(darkgray)&&&&&&&&&&&&&&&&&&&&&&&& \\
    \none &*(darkgray)&*(darkgray)&*(darkgray)&*(darkgray)&*(darkgray)&*(darkgray)&*(darkgray)&&&&&&&&&&&&&&&&&&&&&&&& \\
    \none &*(darkgray)&*(darkgray)&*(darkgray)&*(darkgray)&*(darkgray)&*(darkgray)&*(darkgray)&*(darkgray)&&&&&&&&&&&&&&&&&&&&&&&& \\
    \none &*(darkgray)&*(darkgray)&*(darkgray)&*(darkgray)&*(darkgray)&*(darkgray)&*(darkgray)&*(darkgray)&*(darkgray)&&&&&&&&&&&&&&&&&&&&&&&& \\
    \none &*(darkgray)&*(darkgray)&*(darkgray)&*(darkgray)&*(darkgray)&*(darkgray)&*(darkgray)&*(darkgray)&*(darkgray)&*(darkgray)&&&&&&&&&&&&&&&&&&&&&&&& \\
    \none &*(darkgray)&*(darkgray)&*(darkgray)&*(darkgray)&*(darkgray)&*(darkgray)&*(darkgray)&*(darkgray)&*(darkgray)&*(darkgray)&*(darkgray)&&&&&&&&&&&&&&&&&&&&&&&& \\
    \none &*(darkgray)&*(darkgray)&*(darkgray)&*(darkgray)&*(darkgray)&*(darkgray)&*(darkgray)&*(darkgray)&*(darkgray)&*(darkgray)&*(darkgray)&*(darkgray)&&&&&&&&&&&&&&&&&&&&&&&&\\
    \none &*(darkgray)&*(darkgray)&*(darkgray)&*(darkgray)&*(darkgray)&*(darkgray)&*(darkgray)&*(darkgray)&*(darkgray)&*(darkgray)&*(darkgray)&*(darkgray)&*(darkgray)&&&&&&&&&&&&&&&&&&&&&&&& \\
    \none &*(darkgray)&*(darkgray)&*(darkgray)&*(darkgray)&*(darkgray)&*(darkgray)&*(darkgray)&*(darkgray)&*(darkgray)&*(darkgray)&*(darkgray)&*(darkgray)&*(darkgray)&*(darkgray)&&&&&&&&&&&&&&&&&&&&&&&& \\
    \none &*(darkgray)&*(darkgray)&*(darkgray)&*(darkgray)&*(darkgray)&*(darkgray)&*(darkgray)&*(darkgray)&*(darkgray)&*(darkgray)&*(darkgray)&*(darkgray)&*(darkgray)&*(darkgray)&*(darkgray)&&&&&&&&&&&&&&&&&&&&&&&& \\
    \none &*(darkgray)&*(darkgray)&*(darkgray)&*(darkgray)&*(darkgray)&*(darkgray)&*(darkgray)&*(darkgray)&*(darkgray)&*(darkgray)&*(darkgray)&*(darkgray)&*(darkgray)&*(darkgray)&*(darkgray)&*(darkgray)&&&&&&&&&&&&&&&&&&&&&&&& \\
\end{ytableau}
\quad \quad \quad \quad \quad
\begin{ytableau}
    \none &&&&&&&&&&&&&&&&&&&&&&&&&&&&&&&&&&&&&&&& \\
    \none &&&&&&&&&&&&&&&&&&&&&&&&&&&&&&&&&&&&&&&& \\
    \none &&&&&&&&&&&&&&&&&&&&&&&&&&&&&&&&&&&&&&&& \\
    \none &&&&&&&&&&&&&&&&&&&&&&&&&&&&&&&&&&&&&&&& \\
    \none &&&&&&&&&&&&&&&&&&&&&&&&&&&&&&&&&&&&&&&& \\
    \none &&&&&&&&&&&&&&&&&&&&&&&&&&&&&&&&&&&&&&&& \\
    \none &&&&&&&&&&&&&&&&&&&&&&&&&&&&&&&&&&&&&&&& \\
    \none &&&&&&&&&&&&&&&&&&&&&&&&&&&&&&&&&&&&&&&& \\
    \none &&&&&&&&&&&&&&&&&&&&&&&&&&&&&&&&&&&&&&&& \\
    \none &&&&&&&&&&&&&&&&&&&&&&&&&&&&&&&&&&&&&&&& \\
    \none &&&&&&&&&&&&&&&&&&&&&&&&&&&&&&&&&&&&&&&& \\
    \none &&&&&&&&&&&&&&&&&&&&&&&&&&&&&&&&&&&&&&&& \\
    \none &&&&&&&&&&&&&&&&&&&&&&&&&&&&&&&&&&&&&&&& \\
    \none &&&&&&&&&&&&&&&&&&&&&&&&&&&&&&&&&&&&&&&& \\
    \none &&&&&&&&&&&&&&&&&&&&&&&&&&&&&&&&&&&&&&&& \\
    \none &&&&&&&&&&&&&&&&&&&&&&&&&&&&&&&&&&&&&&&& \\
    \none &&&&&&&&&&&&&&&&&&&&&&&&&&&&&&&&&&&&&&&& \\
    \none &&&&&&&&&&&&&&&&&&&&&&&&&&&&&&&&&&&&&&&& \\
    \none &&&&&&&&&&&&&&&&&&&&&&&&&&&&&&&&&&&&&&&& \\
    \none &&&&&&&&&&&&&&&&&&&&&&&&&&&&&&&&&&&&&&&& \\
    \none &&&&&&&&&&&&&&&&&&&&&&&&&&&&&&&&&&&&&&&& \\
    \none &&&&&&&&&&&&&&&&&&&&&&&&&&&&&&&&&&&&&&&& \\
    \none &&&&&&&&&&&&&&&&&&&&&&&&&&&&&&&&&&&&&&&& \\
    \none &&&&&&&&&&&&&&&&&&&&&&&&&&&&&&&&&&&&&&&& \\
    \none &*(darkgray)&&&&&&&&&&&&&&&&&&&&&&&&&&&&&&&&&&&&&&& \\
    \none &*(darkgray)&*(darkgray)&&&&&&&&&&&&&&&&&&&&&&&&&&&&&&&&&&&&&& \\
    \none &*(darkgray)&*(darkgray)&*(darkgray)&&&&&&&&&&&&&&&&&&&&&&&&&&&&&&&&&&&&& \\
    \none &*(darkgray)&*(darkgray)&*(darkgray)&*(darkgray)&&&&&&&&&&&&&&&&&&&&&&&&&&&&&&&&&&&& \\
    \none &*(darkgray)&*(darkgray)&*(darkgray)&*(darkgray)&*(darkgray)&&&&&&&&&&&&&&&&&&&&&&&&&&&&&&&&&&& \\
    \none &*(darkgray)&*(darkgray)&*(darkgray)&*(darkgray)&*(darkgray)&*(darkgray)&&&&&&&&&&&&&&&&&&&&&&&&&&&&&&&&&& \\
    \none &*(darkgray)&*(darkgray)&*(darkgray)&*(darkgray)&*(darkgray)&*(darkgray)&*(darkgray)&&&&&&&&&&&&&&&&&&&&&&&&&&&&&&&&& \\
    \none &*(darkgray)&*(darkgray)&*(darkgray)&*(darkgray)&*(darkgray)&*(darkgray)&*(darkgray)&*(darkgray)&&&&&&&&&&&&&&&&&&&&&&&&&&&&&&&& \\
    \none &*(darkgray)&*(darkgray)&*(darkgray)&*(darkgray)&*(darkgray)&*(darkgray)&*(darkgray)&*(darkgray)&*(darkgray)&&&&&&&&&&&&&&&&&&&&&&&&&&&&&&& \\
    \none &*(darkgray)&*(darkgray)&*(darkgray)&*(darkgray)&*(darkgray)&*(darkgray)&*(darkgray)&*(darkgray)&*(darkgray)&*(darkgray)&&&&&&&&&&&&&&&&&&&&&&&&&&&&&& \\
    \none &*(darkgray)&*(darkgray)&*(darkgray)&*(darkgray)&*(darkgray)&*(darkgray)&*(darkgray)&*(darkgray)&*(darkgray)&*(darkgray)&*(darkgray)&&&&&&&&&&&&&&&&&&&&&&&&&&&&& \\
    \none &*(darkgray)&*(darkgray)&*(darkgray)&*(darkgray)&*(darkgray)&*(darkgray)&*(darkgray)&*(darkgray)&*(darkgray)&*(darkgray)&*(darkgray)&*(darkgray)&&&&&&&&&&&&&&&&&&&&&&&&&&&& \\
    \none &*(darkgray)&*(darkgray)&*(darkgray)&*(darkgray)&*(darkgray)&*(darkgray)&*(darkgray)&*(darkgray)&*(darkgray)&*(darkgray)&*(darkgray)&*(darkgray)&*(darkgray)&&&&&&&&&&&&&&&&&&&&&&&&&&& \\
    \none &*(darkgray)&*(darkgray)&*(darkgray)&*(darkgray)&*(darkgray)&*(darkgray)&*(darkgray)&*(darkgray)&*(darkgray)&*(darkgray)&*(darkgray)&*(darkgray)&*(darkgray)&*(darkgray)&&&&&&&&&&&&&&&&&&&&&&&&&& \\
    \none &*(darkgray)&*(darkgray)&*(darkgray)&*(darkgray)&*(darkgray)&*(darkgray)&*(darkgray)&*(darkgray)&*(darkgray)&*(darkgray)&*(darkgray)&*(darkgray)&*(darkgray)&*(darkgray)&*(darkgray)&&&&&&&&&&&&&&&&&&&&&&&&& \\
    \none &*(darkgray)&*(darkgray)&*(darkgray)&*(darkgray)&*(darkgray)&*(darkgray)&*(darkgray)&*(darkgray)&*(darkgray)&*(darkgray)&*(darkgray)&*(darkgray)&*(darkgray)&*(darkgray)&*(darkgray)&*(darkgray)&&&&&&&&&&&&&&&&&&&&&&&& \\
\end{ytableau}
}
    \caption{The Young diagrams $\lambda = [40, 39, \ldots, 2, 1]$ and $\lambda = [40^{40}]$. In dark gray are possible positions for non-trivial ribbons, typically clustered if the ribbons are small.}
    \label{fig: corner ribbon peeling}
\end{figure}

We tried this approach in \cite{Teyssier2025SkewDimensionsAndCharacters}. This choice of $\alpha$ leads to an \textit{unsigned} geometry which is tractable, but the triangle inequality in \eqref{eq: MN with triangle inequality and min} loses too much information. Another way to state this is that to improve the bounds with such an $\alpha$ would require understanding the sign compensations in the Murnaghan--Nakayama rule.

\medskip

Another possible choice of $\alpha$ is to give a large index to larger ribbons. Schematically, this gives ribbon tableaux as in Figure~\ref{fig: external ribbon peeling}.

\begin{figure}[!ht]
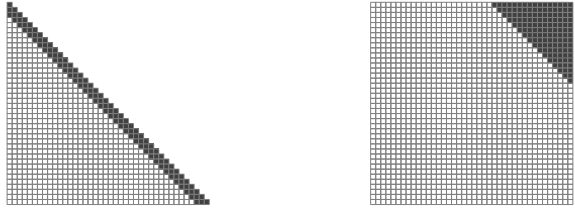

    \centering
{
\ytableausetup{boxsize=0.16em, boxframe = 0.15pt}
\color{black!50}
\begin{ytableau}
    \none &*(darkgray) \\
    \none &*(darkgray)&*(darkgray) \\
    \none &*(darkgray)&*(darkgray)&*(darkgray) \\
    \none &&*(darkgray)&*(darkgray)&*(darkgray) \\
    \none &&&*(darkgray)&*(darkgray)&*(darkgray) \\
    \none &&&&*(darkgray)&*(darkgray)&*(darkgray) \\
    \none &&&&&*(darkgray)&*(darkgray)&*(darkgray) \\
    \none &&&&&&*(darkgray)&*(darkgray)&*(darkgray) \\
    \none &&&&&&&*(darkgray)&*(darkgray)&*(darkgray) \\
    \none &&&&&&&&*(darkgray)&*(darkgray)&*(darkgray) \\
    \none &&&&&&&&&*(darkgray)&*(darkgray)&*(darkgray) \\
    \none &&&&&&&&&&*(darkgray)&*(darkgray)&*(darkgray) \\
    \none &&&&&&&&&&&*(darkgray)&*(darkgray)&*(darkgray) \\
    \none &&&&&&&&&&&&*(darkgray)&*(darkgray)&*(darkgray) \\
    \none &&&&&&&&&&&&&*(darkgray)&*(darkgray)&*(darkgray) \\
    \none &&&&&&&&&&&&&&*(darkgray)&*(darkgray)&*(darkgray) \\
    \none &&&&&&&&&&&&&&&*(darkgray)&*(darkgray)&*(darkgray) \\
    \none &&&&&&&&&&&&&&&&*(darkgray)&*(darkgray)&*(darkgray) \\
    \none &&&&&&&&&&&&&&&&&*(darkgray)&*(darkgray)&*(darkgray) \\
    \none &&&&&&&&&&&&&&&&&&*(darkgray)&*(darkgray)&*(darkgray) \\
    \none &&&&&&&&&&&&&&&&&&&*(darkgray)&*(darkgray)&*(darkgray) \\
    \none &&&&&&&&&&&&&&&&&&&&*(darkgray)&*(darkgray)&*(darkgray) \\
    \none &&&&&&&&&&&&&&&&&&&&&*(darkgray)&*(darkgray)&*(darkgray) \\
    \none &&&&&&&&&&&&&&&&&&&&&&*(darkgray)&*(darkgray)&*(darkgray) \\
    \none &&&&&&&&&&&&&&&&&&&&&&&*(darkgray)&*(darkgray)&*(darkgray) \\
    \none &&&&&&&&&&&&&&&&&&&&&&&&*(darkgray)&*(darkgray)&*(darkgray) \\
    \none &&&&&&&&&&&&&&&&&&&&&&&&&*(darkgray)&*(darkgray)&*(darkgray) \\
    \none &&&&&&&&&&&&&&&&&&&&&&&&&&*(darkgray)&*(darkgray)&*(darkgray) \\
    \none &&&&&&&&&&&&&&&&&&&&&&&&&&&*(darkgray)&*(darkgray)&*(darkgray) \\
    \none &&&&&&&&&&&&&&&&&&&&&&&&&&&&*(darkgray)&*(darkgray)&*(darkgray) \\
    \none &&&&&&&&&&&&&&&&&&&&&&&&&&&&&*(darkgray)&*(darkgray)&*(darkgray) \\
    \none &&&&&&&&&&&&&&&&&&&&&&&&&&&&&&*(darkgray)&*(darkgray)&*(darkgray) \\
    \none &&&&&&&&&&&&&&&&&&&&&&&&&&&&&&&*(darkgray)&*(darkgray)&*(darkgray) \\
    \none &&&&&&&&&&&&&&&&&&&&&&&&&&&&&&&&*(darkgray)&*(darkgray)&*(darkgray) \\
    \none &&&&&&&&&&&&&&&&&&&&&&&&&&&&&&&&&*(darkgray)&*(darkgray)&*(darkgray) \\
    \none &&&&&&&&&&&&&&&&&&&&&&&&&&&&&&&&&&*(darkgray)&*(darkgray)&*(darkgray) \\
    \none &&&&&&&&&&&&&&&&&&&&&&&&&&&&&&&&&&&*(darkgray)&*(darkgray)&*(darkgray) \\
    \none &&&&&&&&&&&&&&&&&&&&&&&&&&&&&&&&&&&&*(darkgray)&*(darkgray)&*(darkgray) \\
    \none &&&&&&&&&&&&&&&&&&&&&&&&&&&&&&&&&&&&&*(darkgray)&*(darkgray)&*(darkgray) \\
    \none &&&&&&&&&&&&&&&&&&&&&&&&&&&&&&&&&&&&&&*(darkgray)&*(darkgray)&*(darkgray) \\
\end{ytableau}
\quad \quad \quad \quad \quad
\begin{ytableau}
    \none &&&&&&&&&&&&&&&&&&&&&&&&&*(darkgray)&*(darkgray)&*(darkgray)&*(darkgray)&*(darkgray)&*(darkgray)&*(darkgray)&*(darkgray)&*(darkgray)&*(darkgray)&*(darkgray)&*(darkgray)&*(darkgray)&*(darkgray)&*(darkgray)&*(darkgray) \\
    \none &&&&&&&&&&&&&&&&&&&&&&&&&&*(darkgray)&*(darkgray)&*(darkgray)&*(darkgray)&*(darkgray)&*(darkgray)&*(darkgray)&*(darkgray)&*(darkgray)&*(darkgray)&*(darkgray)&*(darkgray)&*(darkgray)&*(darkgray)&*(darkgray) \\
    \none &&&&&&&&&&&&&&&&&&&&&&&&&&&*(darkgray)&*(darkgray)&*(darkgray)&*(darkgray)&*(darkgray)&*(darkgray)&*(darkgray)&*(darkgray)&*(darkgray)&*(darkgray)&*(darkgray)&*(darkgray)&*(darkgray)&*(darkgray) \\
    \none &&&&&&&&&&&&&&&&&&&&&&&&&&&&*(darkgray)&*(darkgray)&*(darkgray)&*(darkgray)&*(darkgray)&*(darkgray)&*(darkgray)&*(darkgray)&*(darkgray)&*(darkgray)&*(darkgray)&*(darkgray)&*(darkgray) \\
    \none &&&&&&&&&&&&&&&&&&&&&&&&&&&&&*(darkgray)&*(darkgray)&*(darkgray)&*(darkgray)&*(darkgray)&*(darkgray)&*(darkgray)&*(darkgray)&*(darkgray)&*(darkgray)&*(darkgray)&*(darkgray) \\
    \none &&&&&&&&&&&&&&&&&&&&&&&&&&&&&&*(darkgray)&*(darkgray)&*(darkgray)&*(darkgray)&*(darkgray)&*(darkgray)&*(darkgray)&*(darkgray)&*(darkgray)&*(darkgray)&*(darkgray) \\
    \none &&&&&&&&&&&&&&&&&&&&&&&&&&&&&&&*(darkgray)&*(darkgray)&*(darkgray)&*(darkgray)&*(darkgray)&*(darkgray)&*(darkgray)&*(darkgray)&*(darkgray)&*(darkgray) \\
    \none &&&&&&&&&&&&&&&&&&&&&&&&&&&&&&&&*(darkgray)&*(darkgray)&*(darkgray)&*(darkgray)&*(darkgray)&*(darkgray)&*(darkgray)&*(darkgray)&*(darkgray) \\
    \none &&&&&&&&&&&&&&&&&&&&&&&&&&&&&&&&&*(darkgray)&*(darkgray)&*(darkgray)&*(darkgray)&*(darkgray)&*(darkgray)&*(darkgray)&*(darkgray) \\
    \none &&&&&&&&&&&&&&&&&&&&&&&&&&&&&&&&&&*(darkgray)&*(darkgray)&*(darkgray)&*(darkgray)&*(darkgray)&*(darkgray)&*(darkgray) \\
    \none &&&&&&&&&&&&&&&&&&&&&&&&&&&&&&&&&&&*(darkgray)&*(darkgray)&*(darkgray)&*(darkgray)&*(darkgray)&*(darkgray) \\
    \none &&&&&&&&&&&&&&&&&&&&&&&&&&&&&&&&&&&&*(darkgray)&*(darkgray)&*(darkgray)&*(darkgray)&*(darkgray) \\
    \none &&&&&&&&&&&&&&&&&&&&&&&&&&&&&&&&&&&&&*(darkgray)&*(darkgray)&*(darkgray)&*(darkgray) \\
    \none &&&&&&&&&&&&&&&&&&&&&&&&&&&&&&&&&&&&&&*(darkgray)&*(darkgray)&*(darkgray) \\
    \none &&&&&&&&&&&&&&&&&&&&&&&&&&&&&&&&&&&&&&&*(darkgray)&*(darkgray) \\
    \none &&&&&&&&&&&&&&&&&&&&&&&&&&&&&&&&&&&&&&&&*(darkgray) \\
    \none &&&&&&&&&&&&&&&&&&&&&&&&&&&&&&&&&&&&&&&& \\
    \none &&&&&&&&&&&&&&&&&&&&&&&&&&&&&&&&&&&&&&&& \\
    \none &&&&&&&&&&&&&&&&&&&&&&&&&&&&&&&&&&&&&&&& \\
    \none &&&&&&&&&&&&&&&&&&&&&&&&&&&&&&&&&&&&&&&& \\
    \none &&&&&&&&&&&&&&&&&&&&&&&&&&&&&&&&&&&&&&&& \\
    \none &&&&&&&&&&&&&&&&&&&&&&&&&&&&&&&&&&&&&&&& \\
    \none &&&&&&&&&&&&&&&&&&&&&&&&&&&&&&&&&&&&&&&& \\
    \none &&&&&&&&&&&&&&&&&&&&&&&&&&&&&&&&&&&&&&&& \\
    \none &&&&&&&&&&&&&&&&&&&&&&&&&&&&&&&&&&&&&&&& \\
    \none &&&&&&&&&&&&&&&&&&&&&&&&&&&&&&&&&&&&&&&& \\
    \none &&&&&&&&&&&&&&&&&&&&&&&&&&&&&&&&&&&&&&&& \\
    \none &&&&&&&&&&&&&&&&&&&&&&&&&&&&&&&&&&&&&&&& \\
    \none &&&&&&&&&&&&&&&&&&&&&&&&&&&&&&&&&&&&&&&& \\
    \none &&&&&&&&&&&&&&&&&&&&&&&&&&&&&&&&&&&&&&&& \\
    \none &&&&&&&&&&&&&&&&&&&&&&&&&&&&&&&&&&&&&&&& \\
    \none &&&&&&&&&&&&&&&&&&&&&&&&&&&&&&&&&&&&&&&& \\
    \none &&&&&&&&&&&&&&&&&&&&&&&&&&&&&&&&&&&&&&&& \\
    \none &&&&&&&&&&&&&&&&&&&&&&&&&&&&&&&&&&&&&&&& \\
    \none &&&&&&&&&&&&&&&&&&&&&&&&&&&&&&&&&&&&&&&& \\
    \none &&&&&&&&&&&&&&&&&&&&&&&&&&&&&&&&&&&&&&&& \\
    \none &&&&&&&&&&&&&&&&&&&&&&&&&&&&&&&&&&&&&&&& \\
    \none &&&&&&&&&&&&&&&&&&&&&&&&&&&&&&&&&&&&&&&& \\
    \none &&&&&&&&&&&&&&&&&&&&&&&&&&&&&&&&&&&&&&&& \\
    \none &&&&&&&&&&&&&&&&&&&&&&&&&&&&&&&&&&&&&&&& \\
\end{ytableau}
}
    \caption{The Young diagrams $\lambda = [40, 39, \ldots, 2, 1]$ and $\lambda = [40^{40}]$. In dark gray are possible positions for non-trivial ribbons, typically spread on the left but clustered on the right, if the ribbons are small.}
    \label{fig: external ribbon peeling}
\end{figure}

With this approach, if the diagram $\lambda$ is not \textit{too degenerate}, the non-trivial ribbons spread nicely -- along the \textit{hypotenuse} for triangular shapes -- and make some analysis possible. However, for squares the non-trivial ribbons are clustered, and we are brought to the same issue as in the first choice of $\alpha$.

\subsection{A solution: river ribbon tableaux}

Ideally, we would want the non-trivial ribbons not to be clustered, as in the triangular case in Figure~\ref{fig: external ribbon peeling}, but also when $\lambda$ is a square.
The solution we found is to sandwich the non-trivial ribbons between two blocks of single boxes (i.e.\ ribbons of length 1), which we expect to look schematically as in Figure~\ref{fig: river ribbon peeling}.

\begin{figure}[!ht]
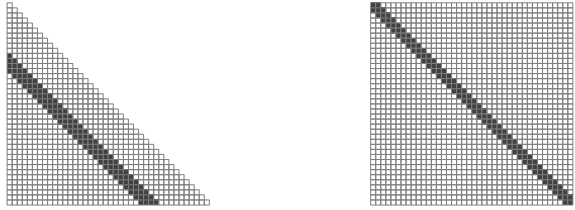

    \centering
{
\ytableausetup{boxsize=0.16em, boxframe = 0.15pt}
\color{black!50}
\begin{ytableau}
    \none & \\
    \none && \\
    \none &&& \\
    \none &&&& \\
    \none &&&&& \\
    \none &&&&&& \\
    \none &&&&&&& \\
    \none &&&&&&&& \\
    \none &&&&&&&&& \\
    \none &&&&&&&&&& \\
    \none &*(darkgray)&&&&&&&&&& \\
    \none &*(darkgray)&*(darkgray)&&&&&&&&&& \\
    \none &*(darkgray)&*(darkgray)&*(darkgray)&&&&&&&&&& \\
    \none &*(darkgray)&*(darkgray)&*(darkgray)&*(darkgray)&&&&&&&&&& \\
    \none &&*(darkgray)&*(darkgray)&*(darkgray)&*(darkgray)&&&&&&&&&& \\
    \none &&&*(darkgray)&*(darkgray)&*(darkgray)&*(darkgray)&&&&&&&&&& \\
    \none &&&&*(darkgray)&*(darkgray)&*(darkgray)&*(darkgray)&&&&&&&&&& \\
    \none &&&&&*(darkgray)&*(darkgray)&*(darkgray)&*(darkgray)&&&&&&&&&& \\
    \none &&&&&&*(darkgray)&*(darkgray)&*(darkgray)&*(darkgray)&&&&&&&&&& \\
    \none &&&&&&&*(darkgray)&*(darkgray)&*(darkgray)&*(darkgray)&&&&&&&&&& \\
    \none &&&&&&&&*(darkgray)&*(darkgray)&*(darkgray)&*(darkgray)&&&&&&&&&& \\
    \none &&&&&&&&&*(darkgray)&*(darkgray)&*(darkgray)&*(darkgray)&&&&&&&&&& \\
    \none &&&&&&&&&&*(darkgray)&*(darkgray)&*(darkgray)&*(darkgray)&&&&&&&&&& \\
    \none &&&&&&&&&&&*(darkgray)&*(darkgray)&*(darkgray)&*(darkgray)&&&&&&&&&& \\
    \none &&&&&&&&&&&&*(darkgray)&*(darkgray)&*(darkgray)&*(darkgray)&&&&&&&&&& \\
    \none &&&&&&&&&&&&&*(darkgray)&*(darkgray)&*(darkgray)&*(darkgray)&&&&&&&&&& \\
    \none &&&&&&&&&&&&&&*(darkgray)&*(darkgray)&*(darkgray)&*(darkgray)&&&&&&&&&& \\
    \none &&&&&&&&&&&&&&&*(darkgray)&*(darkgray)&*(darkgray)&*(darkgray)&&&&&&&&&& \\
    \none &&&&&&&&&&&&&&&&*(darkgray)&*(darkgray)&*(darkgray)&*(darkgray)&&&&&&&&&& \\
    \none &&&&&&&&&&&&&&&&&*(darkgray)&*(darkgray)&*(darkgray)&*(darkgray)&&&&&&&&&& \\
    \none &&&&&&&&&&&&&&&&&&*(darkgray)&*(darkgray)&*(darkgray)&*(darkgray)&&&&&&&&&& \\
    \none &&&&&&&&&&&&&&&&&&&*(darkgray)&*(darkgray)&*(darkgray)&*(darkgray)&&&&&&&&&& \\
    \none &&&&&&&&&&&&&&&&&&&&*(darkgray)&*(darkgray)&*(darkgray)&*(darkgray)&&&&&&&&&& \\
    \none &&&&&&&&&&&&&&&&&&&&&*(darkgray)&*(darkgray)&*(darkgray)&*(darkgray)&&&&&&&&&& \\
    \none &&&&&&&&&&&&&&&&&&&&&&*(darkgray)&*(darkgray)&*(darkgray)&*(darkgray)&&&&&&&&&& \\
    \none &&&&&&&&&&&&&&&&&&&&&&&*(darkgray)&*(darkgray)&*(darkgray)&*(darkgray)&&&&&&&&&& \\
    \none &&&&&&&&&&&&&&&&&&&&&&&&*(darkgray)&*(darkgray)&*(darkgray)&*(darkgray)&&&&&&&&&& \\
    \none &&&&&&&&&&&&&&&&&&&&&&&&&*(darkgray)&*(darkgray)&*(darkgray)&*(darkgray)&&&&&&&&&& \\
    \none &&&&&&&&&&&&&&&&&&&&&&&&&&*(darkgray)&*(darkgray)&*(darkgray)&*(darkgray)&&&&&&&&&& \\
    \none &&&&&&&&&&&&&&&&&&&&&&&&&&&*(darkgray)&*(darkgray)&*(darkgray)&*(darkgray)&&&&&&&&&& \\
\end{ytableau}
\quad \quad \quad \quad \quad
\begin{ytableau}
    \none &*(darkgray)&*(darkgray)&&&&&&&&&&&&&&&&&&&&&&&&&&&&&&&&&&&&&& \\
    \none &*(darkgray)&*(darkgray)&*(darkgray)&&&&&&&&&&&&&&&&&&&&&&&&&&&&&&&&&&&&& \\
    \none &&*(darkgray)&*(darkgray)&*(darkgray)&&&&&&&&&&&&&&&&&&&&&&&&&&&&&&&&&&&& \\
    \none &&&*(darkgray)&*(darkgray)&*(darkgray)&&&&&&&&&&&&&&&&&&&&&&&&&&&&&&&&&&& \\
    \none &&&&*(darkgray)&*(darkgray)&*(darkgray)&&&&&&&&&&&&&&&&&&&&&&&&&&&&&&&&&& \\
    \none &&&&&*(darkgray)&*(darkgray)&*(darkgray)&&&&&&&&&&&&&&&&&&&&&&&&&&&&&&&&& \\
    \none &&&&&&*(darkgray)&*(darkgray)&*(darkgray)&&&&&&&&&&&&&&&&&&&&&&&&&&&&&&&& \\
    \none &&&&&&&*(darkgray)&*(darkgray)&*(darkgray)&&&&&&&&&&&&&&&&&&&&&&&&&&&&&&& \\
    \none &&&&&&&&*(darkgray)&*(darkgray)&*(darkgray)&&&&&&&&&&&&&&&&&&&&&&&&&&&&&& \\
    \none &&&&&&&&&*(darkgray)&*(darkgray)&*(darkgray)&&&&&&&&&&&&&&&&&&&&&&&&&&&&& \\
    \none &&&&&&&&&&*(darkgray)&*(darkgray)&*(darkgray)&&&&&&&&&&&&&&&&&&&&&&&&&&&& \\
    \none &&&&&&&&&&&*(darkgray)&*(darkgray)&*(darkgray)&&&&&&&&&&&&&&&&&&&&&&&&&&& \\
    \none &&&&&&&&&&&&*(darkgray)&*(darkgray)&*(darkgray)&&&&&&&&&&&&&&&&&&&&&&&&&& \\
    \none &&&&&&&&&&&&&*(darkgray)&*(darkgray)&*(darkgray)&&&&&&&&&&&&&&&&&&&&&&&&& \\
    \none &&&&&&&&&&&&&&*(darkgray)&*(darkgray)&*(darkgray)&&&&&&&&&&&&&&&&&&&&&&&& \\
    \none &&&&&&&&&&&&&&&*(darkgray)&*(darkgray)&*(darkgray)&&&&&&&&&&&&&&&&&&&&&&& \\
    \none &&&&&&&&&&&&&&&&*(darkgray)&*(darkgray)&*(darkgray)&&&&&&&&&&&&&&&&&&&&&& \\
    \none &&&&&&&&&&&&&&&&&*(darkgray)&*(darkgray)&*(darkgray)&&&&&&&&&&&&&&&&&&&&& \\
    \none &&&&&&&&&&&&&&&&&&*(darkgray)&*(darkgray)&*(darkgray)&&&&&&&&&&&&&&&&&&&& \\
    \none &&&&&&&&&&&&&&&&&&&*(darkgray)&*(darkgray)&*(darkgray)&&&&&&&&&&&&&&&&&&& \\
    \none &&&&&&&&&&&&&&&&&&&&*(darkgray)&*(darkgray)&*(darkgray)&&&&&&&&&&&&&&&&&& \\
    \none &&&&&&&&&&&&&&&&&&&&&*(darkgray)&*(darkgray)&*(darkgray)&&&&&&&&&&&&&&&&& \\
    \none &&&&&&&&&&&&&&&&&&&&&&*(darkgray)&*(darkgray)&*(darkgray)&&&&&&&&&&&&&&&& \\
    \none &&&&&&&&&&&&&&&&&&&&&&&*(darkgray)&*(darkgray)&*(darkgray)&&&&&&&&&&&&&&& \\
    \none &&&&&&&&&&&&&&&&&&&&&&&&*(darkgray)&*(darkgray)&*(darkgray)&&&&&&&&&&&&&& \\
    \none &&&&&&&&&&&&&&&&&&&&&&&&&*(darkgray)&*(darkgray)&*(darkgray)&&&&&&&&&&&&& \\
    \none &&&&&&&&&&&&&&&&&&&&&&&&&&*(darkgray)&*(darkgray)&*(darkgray)&&&&&&&&&&&& \\
    \none &&&&&&&&&&&&&&&&&&&&&&&&&&&*(darkgray)&*(darkgray)&*(darkgray)&&&&&&&&&&& \\
    \none &&&&&&&&&&&&&&&&&&&&&&&&&&&&*(darkgray)&*(darkgray)&*(darkgray)&&&&&&&&&& \\
    \none &&&&&&&&&&&&&&&&&&&&&&&&&&&&&*(darkgray)&*(darkgray)&*(darkgray)&&&&&&&&& \\
    \none &&&&&&&&&&&&&&&&&&&&&&&&&&&&&&*(darkgray)&*(darkgray)&*(darkgray)&&&&&&&& \\
    \none &&&&&&&&&&&&&&&&&&&&&&&&&&&&&&&*(darkgray)&*(darkgray)&*(darkgray)&&&&&&& \\
    \none &&&&&&&&&&&&&&&&&&&&&&&&&&&&&&&&*(darkgray)&*(darkgray)&*(darkgray)&&&&&& \\
    \none &&&&&&&&&&&&&&&&&&&&&&&&&&&&&&&&&*(darkgray)&*(darkgray)&*(darkgray)&&&&& \\
    \none &&&&&&&&&&&&&&&&&&&&&&&&&&&&&&&&&&*(darkgray)&*(darkgray)&*(darkgray)&&&& \\
    \none &&&&&&&&&&&&&&&&&&&&&&&&&&&&&&&&&&&*(darkgray)&*(darkgray)&*(darkgray)&&& \\
    \none &&&&&&&&&&&&&&&&&&&&&&&&&&&&&&&&&&&&*(darkgray)&*(darkgray)&*(darkgray)&& \\
    \none &&&&&&&&&&&&&&&&&&&&&&&&&&&&&&&&&&&&&*(darkgray)&*(darkgray)&*(darkgray)& \\
    \none &&&&&&&&&&&&&&&&&&&&&&&&&&&&&&&&&&&&&&*(darkgray)&*(darkgray)&*(darkgray) \\
    \none &&&&&&&&&&&&&&&&&&&&&&&&&&&&&&&&&&&&&&&*(darkgray)&*(darkgray) \\
\end{ytableau}
}
    \caption{The Young diagrams $\lambda = [40, 39, \ldots, 2, 1]$ and $\lambda = [40^{40}]$. In dark gray are possible positions for non-trivial ribbons, typically spread on both sides if the ribbons are small.}
    \label{fig: river ribbon peeling}
\end{figure}

We think of the non-trivial ribbons as a river, and the single boxes as banks of the river. This motivates the following definition.

Let $n,p,q,k\geq 0$ be integers such that $p+k+q = n$, and let $r_1,\ldots, r_m\geq 2$ be integers such that $\sum_{i=1}^m r_i = k$. Let $\lambda \vdash n$. A \textbf{river ribbon tableau} with parameters $p,q, (r_i)_{1\leq i\leq m}$ is a ribbon tableau of $\RT(\lambda, \alpha)$, with $\alpha = (1^q, r_1, \ldots, r_m, 1^p)$.

\section{Some useful estimates}
\subsection{Number of peelable and addable ribbons}

We will need to count in how many ways ribbons can be removed or added to a given shape.
The first part of the following lemma is essentially  \cite[Lemma 3.5 (b)]{Teyssier2025SkewDimensionsAndCharacters} and the second extends it. The idea of the proof is the same, but we now present it with the metaphor of a snake climbing down a mountain, as in Figure~\ref{fig: snake metaphor}.

\begin{figure}[!ht]
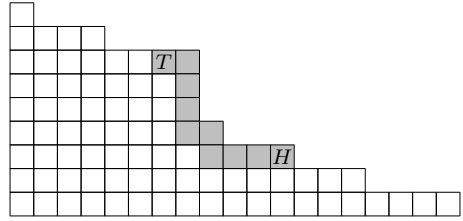

    \centering
{
\ytableausetup{boxsize= 0.8em}
\begin{ytableau}
    \none &\\
    \none &&&&  \\
    \none &&&&&&&*(lightgray) {\scriptstyle T} &*(lightgray) \\
    \none &&&&&&&&*(lightgray) \\
    \none &&&&&&&&*(lightgray) \\
    \none &&&&&&&&*(lightgray)&*(lightgray) \\
    \none &&&&&&&&&*(lightgray)&*(lightgray)&*(lightgray)&*(lightgray) {\scriptstyle H}\\
    \none &&&&&&&&&&&&&&& \\
    \none &&&&&&&&&&&&&&&&&&& \\
\end{ytableau}
}
    \caption{A ribbon, represented in light gray, which may be thought as being either added or removed from a diagram. We see the ribbon as a snake climbing down a mountain. We fill the head and the tail of the snake with the letters $H$ and $T$, respectively.}
    \label{fig: snake metaphor}
\end{figure}

\begin{lemma}\label{lem: bound number of peelable and addable ribbons}
    Let $\eta\subseteq \lambda$ be two Young diagrams, and let $r\geq 2$. Denote again the diagonal length of $\lambda$ by $\delta = \delta(\lambda)$.
\begin{enumerate}
    \item There are at most $2\delta$ ways to remove an $r$-ribbon from $\eta$.
    \item There are at most $2\delta + r$ ways to add an $r$-ribbon to $\eta$, while staying in $\lambda$.
\end{enumerate}
\end{lemma}
\begin{proof}
Each peelable ribbon has either its tail at the top of one of the first $\delta(\eta)$ columns, or its head at the end of one of the first $\delta(\eta)$ rows. Hence there are at most $2\delta(\eta)\leq 2\delta(\lambda)$ such ribbons, proving (a).

Let us now prove (b). There are three different ways to add a ribbon of length $r$. \\
(1) Choose $1\leq i\leq \delta(\eta)$ and place the tail of the ribbon at the end of the $i$-th row. \\
(2) Choose $1\leq j\leq \delta(\eta)$ and place the head of the ribbon at the top of the $j$-th column. \\
(3) Choose a box of the ribbon and place it in position $(\delta(\eta)+1, \delta(\eta)+1)$. 

In all three cases, the chosen box determines the position of the whole ribbon. There are therefore at most $2\delta(\eta) + r\leq 2\delta + r$ such possibilities, which concludes the proof.
\end{proof}

\subsection{Quotients of skew dimensions}
In this section we write $f(\eta) = d_{\eta}/|\eta|!$ if $\eta$ is a (skew) diagram, where $d_\eta$ is the number of standard tableaux of $\eta$, which we call \textit{dimension} of $\eta$ by extension. 
We start by stating the following generalized Björner inequality due to Chan and Pak. See \cite{Bjorner2011qAnalogueFKG} for the original inequality. 

\begin{proposition}[{\cite[Corollary 3.3]{ChanPak2023MultivariateInequalities}}]\label{prop: generalized Bjorner inequality}
Let $\mu \backslash \alpha$ and $\nu\backslash \beta$ be skew Young diagrams. Then $f((\mu\cup \nu) \backslash (\alpha\cup\beta)) \cdot f((\mu\cap\nu)\backslash(\alpha\cap\beta)) \geq f(\mu \backslash \alpha)\cdot f(\nu \backslash \beta)$.
\end{proposition}

Assuming some inclusions between diagrams, Proposition~\ref{prop: generalized Bjorner inequality} simplifies as follows.
\begin{lemma}\label{lem: generalized Bjorner inequality rewritten with mu nu lambda R}
    Let $\mu \subseteq \nu \subseteq \lambda \subseteq R$ be four Young diagrams. Then
\begin{equation}
    \frac{f(\lambda \backslash \nu)}{f(\lambda \backslash \mu)}  \leq \frac{f(R \backslash \nu)}{f(R\backslash \mu)}.
\end{equation}
\end{lemma}
\begin{proof}
We have $f((\lambda \cup R)\backslash(\nu\cup \mu)) \cdot f((\lambda \cap R)\backslash(\nu\cap \mu)) \geq f(\lambda \backslash \nu)\cdot f(R \backslash \mu)$ by Proposition~\ref{prop: generalized Bjorner inequality} applied with $ (\nu, \mu, \lambda, R)$ in place of $(\alpha, \beta, \mu, \nu)$.  Since $\mu \subseteq \nu \subseteq \lambda\subseteq R$ by assumption, this simplifies to $f(R\backslash\nu) \cdot f(\lambda \backslash \mu) \geq f(\lambda \backslash \nu)\cdot f(R \backslash \mu)$, which concludes the proof.
\end{proof}

We denote the content of a box $u$ on the $i$-th row and $j$-th column of a diagram by $c(u) = j-i$. The following lemma is a reformulation of a result of Stanley \cite{Stanley2003CharactersRectangularShape}, itself based on observations of Regev \cite{Regev2000GeneralizedHookAndContentNumbersIdentities}.
\begin{lemma}\label{lem: identity Stanley}
    Let $a,b\geq 1$ be integers. Let $R =[a^b]$ be a rectangle, and let $\eta \subseteq R$ be a Young diagram.
    Then
\begin{equation}
    f(R\backslash \eta) = f(R) \frac{d_\eta}{|\eta|!} \prod_{u\in \eta}(b+c(u))(a-c(u)).
\end{equation}
\end{lemma}
\begin{proof}
    As in \cite{Stanley2003CharactersRectangularShape}, we let $\tilde{\eta}$ be the diagram $R\backslash \eta$, rotated by $\pi$, and we denote the product of all hook lengths of a diagram $\xi$ by $H_\xi$. Then, combining Equations (4) and (5) from \cite{Stanley2003CharactersRectangularShape} gives $H_\eta H_R =  H_{\tilde{\eta}} \prod_{u\in \eta}(b+c(u))(a-c(u))$.
Moreover, we have $d_{R\backslash \eta} = d_{\tilde{\eta}}$ by symmetry. Therefore, by the hook length formula, we have $H_R = |R|!/d_R=1/f(R)$, $H_\eta = |\eta|!/d_\eta$, and $H_{\tilde{\eta}} = (|R|-|\eta|)!/d_{\tilde{\eta}} = (|R|-|\eta|)!/d_{R\backslash \eta}= 1/f(R\backslash \eta)$. The result follows.
\end{proof}

We can now prove our main estimate on quotients of skew dimensions.
\begin{proposition}\label{prop: bound on quotients of skew dimensions}
        Let $\mu \subseteq \nu \subseteq \lambda$ be three Young diagrams. Denote $q = |\mu|$, $n=|\lambda|$, $k = |\nu|-|\mu|$, and $p = |\lambda\backslash \nu| = n-k-q$. Then
\begin{equation}
    \frac{d_\mu d_{\lambda \backslash \nu}}{d_\nu d_{\lambda \backslash \mu}} \leq \pg \frac{s+1}{2}\pd^{2k}\frac{p!q!}{(p+k)!(q+k)!} .
\end{equation}
\end{proposition}
\begin{proof}
Let $R = [a^b]$, where $a$ (resp.\ $b$) is the number of columns (resp.\ rows) of $\lambda$. By Lemmas~\ref{lem: generalized Bjorner inequality rewritten with mu nu lambda R} and~\ref{lem: identity Stanley}, we have 
\begin{equation}
    \frac{(p+k)!}{p!}\frac{d_{\lambda \backslash \nu}}{d_{\lambda \backslash \mu}} =  \frac{f(\lambda \backslash \nu)}{f(\lambda \backslash \mu)}  \leq \frac{f(R \backslash \nu)}{f(R\backslash \mu)} = \frac{d_\nu/(q+k)!}{d_{\mu}/q!} \prod_{u\in \nu\backslash
    \mu}(b+c(u))(a-c(u)).
\end{equation}
Reordering the terms gives 
\begin{equation}\label{eq: bound quotients of skew dimensions with contents}
    \frac{d_\mu d_{\lambda \backslash \nu}}{d_\nu d_{\lambda \backslash \mu}} \leq \frac{p!q!}{(p+k)!(q+k)!} \prod_{u\in \nu\backslash
    \mu}(b+c(u))(a-c(u)).
\end{equation}
But for each box $u \in \nu\backslash\mu$,  setting $L = (a+b)/2=(s+1)/2$ and $x = c(u) +(b-a)/2$, we have 
    $(b+c(u))(a-c(u)) = (L-x)(L+x) \leq L^2 = ((s+1)/2)^2$, so the product is upper bounded by $((s+1)/2)^{2k}$. Plugging this into \eqref{eq: bound quotients of skew dimensions with contents} concludes the proof. 
\end{proof}
\section{Bounds on the number of river ribbon tableaux}\label{s: Bounds on the number of river ribbon tableaux}
We first fix some notation for this section.
Let $n\geq 1$, $\lambda\vdash n$, and $\sigma \in \kS_n$. Let $k$ denote the support size of $\sigma$, and let $r_1\leq  \ldots \leq r_m$ be the lengths of the $m$ non-trivial cycles of $\sigma$. Let $p,q\geq 0$ such that $p+q = n-k$. Set 
\begin{equation}
    M_{p,q} = M_{p,q}(\lambda, \sigma) = \bg \RT(\lambda, (1^q, r_1, \ldots, r_m, 1^p)) \bd.
\end{equation}
Let us also set, for $\mu \vdash q$ and $\nu \vdash q + k$ such that $\mu \subseteq \nu \subseteq \lambda$,
\begin{equation}
W_{\mu, \nu} = W_{\mu, \nu}(\lambda, \sigma) = \bg \RT(\nu\backslash \mu, (r_1, \ldots, r_m)) \bd.
\end{equation}
In words, $M_{p,q}$ is the number of river ribbon tableaux of $\lambda$ with weight $(1^q, r_1, \ldots, r_m, 1^p)$, and $W_{\mu, \nu}$ is the number of ways to place the non-trivial ribbons, after having already chosen where we place the single boxes.
\begin{lemma}\label{lem: bound on ribbon covering of the river starting from mu or nu}
  We have
\begin{equation}
    \sum_{\mu, \nu} d_\mu d_{\lambda\backslash \mu} W_{\mu, \nu} \leq  e^{k/(2\delta)}(2\delta)^m d_\lambda \quad \quad \text{ and } \quad \quad  \sum_{\mu, \nu} d_\nu d_{\lambda\backslash \nu} W_{\mu, \nu} \leq (2\delta)^m d_\lambda \, ,
\end{equation}
where the sums are over all $\mu\vdash q$ and $\nu \vdash q+k$ such that $\mu \subseteq \nu \subseteq \lambda$.
\end{lemma}
\begin{proof}
    For each fixed $\mu$, $\sum_\nu W_{\mu,\nu}$ is the number of ways to iteratively add ribbons of length $r_1,\ldots, r_m$ to $\mu$ up to reaching a sub-diagram of $\lambda$ of size $q+k$, which  by Lemma~\ref{lem: bound number of peelable and addable ribbons} is upper bounded by $\prod_{i=1}^m (2\delta + r_i)$. Therefore
\begin{equation}
    \sum_{\mu, \nu} d_\mu d_{\lambda\backslash \mu} W_{\mu, \nu} = \sum_{\mu} d_\mu d_{\lambda\backslash \mu} \sum_\nu W_{\mu, \nu} \leq \pg \prod_{i=1}^m (2\delta + r_i)\pd \sum_{\mu} d_\mu d_{\lambda\backslash \mu}.
\end{equation}
Moreover, $\sum_{\mu} d_\mu d_{\lambda\backslash \mu} = d_\lambda$, and $(2\delta)^{-m}\prod_{i=1}^m (2\delta + r_i) \leq  \prod_{i=1}^m e^{r_i/(2\delta)} = e^{k/(2\delta)}$. This proves the first point. The same arguments prove the second point: we have
\begin{equation}
     \sum_{\mu, \nu} d_\nu d_{\lambda\backslash \nu} W_{\mu, \nu} = \sum_{\nu} d_\nu d_{\lambda\backslash \nu} \sum_\mu W_{\mu, \nu} \leq \pg \prod_{i=1}^m (2\delta)\pd \sum_{\nu} d_\nu d_{\lambda\backslash \nu} =  d_\lambda  (2\delta)^m. \qedhere
\end{equation}
\end{proof}

We now assume that $k\geq 2$, and fix the values of $p$ and $q$ so that the two river banks have the same size (up to parity): we set
\begin{equation}
    p = \lf \frac{n-k}{2}\rf \quad \quad \text{ and } \quad \quad q = \lc \frac{n-k}{2}\rc.
\end{equation}

\begin{lemma}\label{lem: technical bound p q with factorials}
 We have
 \begin{equation}
      \sqrt{\frac{p!q!}{(p+k)!(q+k)!}} \leq (e/n)^k.
 \end{equation}
\end{lemma}
\begin{proof}
    Rewriting the products, we have $x:= \frac{(p+k)!(q+k)!}{p!q!} = \prod_{j=p+1}^{p+k}j(n+1-j)$. Taking logarithms gives $\ln \sqrt{x} = \sum_{j=p+1}^{p+k}a_j$, where $a_j = \frac{1}{2} \ln j(n+1-j)$ for $1\leq j \leq n$.
But the $a_j$ for $j$ from $p+1$ to $p+k$ are the $k$ maximal terms of the sum $\sum_{j=1}^{n}a_j$. Therefore, lower bounding by the average, we obtain that
\begin{equation}
    \ln \sqrt{x} \geq \frac{k}{n}\sum_{j=1}^n a_j = \frac{k}{n}\ln n!.
\end{equation}
This implies that $\sqrt{1/x} \leq n!^{-k/n}$. Using finally that $n!\geq (n/e)^n$ concludes the proof.
\end{proof}

Our main bound on the number of river ribbon tableaux is the following.
\begin{proposition}\label{prop: bound number river ribbon tableaux directly with specific p and q}
Set $\gamma = (e/2) e^{1/(4\delta)}(1+1/s)$. Then
\begin{equation}
    \frac{M_{p,q}}{d_\lambda} \leq  (\gamma s/n)^k (2\delta)^m.
\end{equation}
\end{proposition}
\begin{proof}
By definition we have $M_{p,q} = \sum_{\mu, \nu} d_\mu W_{\mu, \nu} d_{\lambda\backslash \nu}$, where the sum is over all $\mu\vdash q$ and $\nu \vdash q+k$ such that $\mu \subseteq \nu \subseteq \lambda$.
Reorganizing this sum, applying Proposition~\ref{prop: bound on quotients of skew dimensions}, and using the Cauchy--Schwarz inequality, we have
\begin{equation}
\begin{split}
    M_{p,q} & = \sum_{\mu, \nu} \sqrt{d_\mu d_{\lambda\backslash \mu} W_{\mu, \nu}} \sqrt{d_\nu d_{\lambda\backslash \nu}W_{\mu, \nu}}   \sqrt{\frac{d_\mu d_{\lambda\backslash \nu}}{d_\nu d_{\lambda\backslash \mu}}} \\
    & \leq \pg \frac{s+1}{2}\pd^{k}\sqrt{\frac{p!q!}{(p+k)!(q+k)!}} \sum_{\mu, \nu} \sqrt{d_\mu d_{\lambda\backslash \mu} W_{\mu, \nu}} \sqrt{d_\nu d_{\lambda\backslash \nu}W_{\mu, \nu}} \\
    & \leq \pg \frac{s+1}{2}\pd^{k}\sqrt{\frac{p!q!}{(p+k)!(q+k)!}}  \sqrt{\sum_{\mu, \nu} d_\mu d_{\lambda\backslash \mu} W_{\mu, \nu}} \sqrt{\sum_{\mu, \nu}d_\nu d_{\lambda\backslash \nu}W_{\mu, \nu}}.
\end{split}
\end{equation}
Applying now Lemmas~\ref{lem: bound on ribbon covering of the river starting from mu or nu} and~\ref{lem: technical bound p q with factorials} and reordering the terms conclude the proof. 
\end{proof}

\section{Character bounds}
We start by proving a stronger, but notationally heavier, version of Theorem~\ref{thm: main general form}, and then prove Theorems~\ref{thm: main general form} and~\ref{thm: main balanced Moore Russell form}.
\begin{theorem}\label{thm: main general form with m}
Let $n\geq 1$, $\lambda\vdash n$ and $\sigma\in \kS_n$. Then
\begin{equation}\label{eq: main general form with m}
     \bg \chi^\lambda(\sigma) \bd \leq \pg \gamma s/n \pd^{k}(2\delta)^{m},
\end{equation}
where $s$ is the maximal hook length of $\lambda$, $\delta$ is the diagonal length of $\lambda$, $k$ is the support size of $\sigma$, $m$ is the number of non-trivial cycles of $\sigma$, and $\gamma = (e/2) e^{1/(4\delta)}(1+1/s)$. In particular, we have
\begin{equation}
    \bg \chi^\lambda(\sigma) \bd \leq (2\gamma^2 \delta s^2/n^2)^{|\sigma|}.
\end{equation}

\end{theorem}
\begin{proof}
    We may assume that $k\geq 2$, since the bound is trivial for $\sigma = \Id$. The first bound then follows immediately from \eqref{eq: MN with triangle inequality and min} and Proposition~\ref{prop: bound number river ribbon tableaux directly with specific p and q}. For the second part, observe that $k=|\sigma|+m$, and [$m\leq |\sigma|$ and $2\gamma \delta s/n\geq \delta s/n \geq 1$], so
    \begin{equation}
        \bg \chi^\lambda(\sigma) \bd \leq \pg \gamma s/n \pd^{k}(2\delta)^{m} = \pg \gamma s/n \pd^{|\sigma|}(2\gamma\delta s/n)^{m}\leq (2\gamma^2 \delta s^2/n^2)^{|\sigma|}. \qedhere
    \end{equation}
\end{proof}

Using the trivial bounds $s\geq 1$ and $\delta \geq 1$ already gives $2\gamma^2 \leq 2e^{5/2} \leq 25$.

Let us optimize the constant a little more.

\begin{proof}[Proof of Theorem~\ref{thm: main general form}]
   Our goal is to prove that $\bg \chi^\lambda(\sigma) \bd \leq (5 \delta s^2/n^2)^{|\sigma|}$. Recall that we always have $\bg \chi^\lambda(\sigma) \bd\leq 1$ and $\delta s \geq n$. This implies that the bound is trivial if $s\geq n/5$. We may therefore assume that $s<n/5$. Plugging this into the inequality $\delta s \geq n$ then gives $\delta>5$, so $\delta \geq 6$ since it is an integer. Therefore $s\geq 2\delta - 1 \geq 11$. Using this we obtain a new bound on $2\gamma^2$: we have $2\gamma^2 \leq (1/2)e^{25/12}(12/11)^2 \leq 5$. Applying Theorem~\ref{thm: main general form with m} with this new bound concludes the proof.
\end{proof}

\begin{proof}[Proof of Theorem~\ref{thm: main balanced Moore Russell form}]
    This follows from Theorem~\ref{thm: main general form}, since for $C$-balanced diagrams we have $\delta s^2/n^2 \leq \sqrt{n} (C\sqrt{n})^2/n^2 = C^2/\sqrt{n}$.
\end{proof}

\section{Application to the Dousse--Féray conjecture}

We now present an application of Theorem~\ref{thm: main balanced Moore Russell form} to the number of standard tableaux of skew shapes. Upper bounds on the quantity $A_{\lambda, \mu} := |\mu|! \frac{d_{\lambda\backslash \mu}}{d_\lambda d_\mu}$ were obtained by Dousse and Féray \cite{DousseFeray2019skewdiagramscharacters} via characters. Based on numerical evidence, they conjectured (see \cite[Conjecture 6]{DousseFeray2019skewdiagramscharacters}) that $|\ln A_{\lambda, \mu}| = O \pg |\mu|^{3/2}/|\lambda|^{1/2}\pd$ holds uniformly for balanced diagrams $\lambda$ and $\mu$. The following theorem confirms their conjecture.

\begin{theorem}\label{thm: skew diagrams upper bound Dousse Feray conjecture}
Let $C>0$. There exists a constant $K$, depending only on $C$, such that for every (non-empty) $C$-balanced Young diagrams $\lambda$ and $\mu$ such that $\mu \subseteq \lambda$, we have
\begin{equation}
   -K \frac{|\mu|^{3/2}}{|\lambda|^{1/2}}\leq \ln \pg |\mu|! \frac{d_{\lambda\backslash \mu}}{d_\lambda d_\mu} \pd \leq K \frac{|\mu|^{3/2}}{|\lambda|^{1/2}}.
\end{equation}
\end{theorem}
For the upper bound, we only need to plug Theorem~\ref{thm: main balanced Moore Russell form} into the estimates of \cite{DousseFeray2019skewdiagramscharacters}.
\begin{proof}[Proof of Theorem~\ref{thm: skew diagrams upper bound Dousse Feray conjecture} (upper bound part)]
    Let $\mu \subseteq \lambda$ be two $C$-balanced Young diagrams. Denote $k = |\mu|$ and $n=|\lambda|$. Applying the triangle inequality and Theorem~\ref{thm: main balanced Moore Russell form} to \cite[Lemma 10]{DousseFeray2019skewdiagramscharacters} (originally from \cite{Stanley2003EnumerationSkewYoungTableaux}), we obtain, setting $B = 5C^2$,
    \begin{equation}
        A_{\lambda, \mu} \leq \sum_{\sigma\in\kS_k} | \chi^\lambda(\sigma_+) | | \chi^\mu(\sigma) | \leq \sum_{\sigma\in\kS_k} \pg \frac{B^2}{\sqrt{nk}}\pd^{|\sigma|},
    \end{equation}
where $\sigma_+$ is the extension of $\sigma$ to $\kS_n$ by fixing $\ag k+1, \ldots, n\ad$. Decomposing according to the value of $|\sigma|$ and applying \cite[Lemma 13]{DousseFeray2019skewdiagramscharacters} (originally from \cite{FeraySniady2011}) then gives
\begin{equation}
    A_{\lambda, \mu} \leq \sum_{i=0}^{k-1} \frac{k^{2i}}{i!} \pg \frac{B^2}{\sqrt{nk}}\pd^{i} \leq \sum_{i=0}^{\infty} \frac{(B^2k^{3/2}/n^{1/2})^i}{i!} = \exp\pg B^2\frac{k^{3/2}}{n^{1/2}} \pd. \qedhere
\end{equation}
\end{proof}
The lower bound requires other arguments. We will prove it using the Naruse hook length formula, see \cite{Naruse2014, MoralesPakPanova2018I}.
We follow the notation from \cite{Teyssier2025SkewDimensionsAndCharacters}. Let $\lambda\vdash n$. Given a box $u \in \lambda$, $H(\lambda, u)$ is the hook length of $u$ in $\lambda$. Given a subset of boxes $E\subseteq \lambda$, we set $H(\lambda, E) = \prod_{u\in E} H(\lambda, u)$. Given a sub-diagram $\mu \subseteq \lambda$, we denote the set of excited diagrams of $\mu$ in $\lambda$ by $\cE(\lambda, \mu)$. Finally, given two integers $0\leq k\leq n$, we use the notation $n^{\downarrow k} = n!/(n-k)!$ for falling factorials. Then, given two Young diagrams $\mu \subseteq \lambda$ the Naruse hook length formula can be written as 
\begin{equation}
\frac{d_{\lambda\backslash\mu}}{d_\lambda} =\frac{S(\lambda, \mu)}{|\lambda|^{\downarrow |\mu|}}, \quad \quad \text{ where } \quad \quad S(\lambda, \mu) = \sum_{E\in \cE(\lambda, \mu)} H(\lambda, E).
\end{equation}

The lower bound of Theorem~\ref{thm: skew diagrams upper bound Dousse Feray conjecture} is easy to obtain if the diagram $\mu$ is large.
\begin{lemma}\label{lem: lower bound Dousse Feray k order n}
    Let $C>0$. There exists a constant $K_1$, depending only on $C$, such that the following holds. Let $\mu \subseteq \lambda$ be two (non-empty) $C$-balanced Young diagrams such that $|\mu| \geq |\lambda|/C^2$. Then $\ln A_{\lambda, \mu}  \geq - K_1 |\mu|^{3/2}/|\lambda|^{1/2}$.
\end{lemma}
\begin{proof}
Write $n=|\lambda|$ and $k=|\mu|$. By the Naruse hook length formula and since $\mu\in \cE(\lambda, \mu)$, we have $\frac{d_{\lambda\backslash\mu}}{d_\lambda} \geq \frac{H(\mu,\mu)}{n^{\downarrow k}}$. It therefore follows from the standard hook length formula and using that $d_\mu^2 \leq \sum_{\nu\vdash k} d_\nu^2 = k!$,  that
\begin{equation}
    A_{\lambda, \mu} \geq \frac{k!}{d_\mu} \frac{H(\mu,\mu)}{n^{\downarrow k}}  = \frac{k!}{d_\mu^2} \frac{k!}{n^{\downarrow k}} \geq \frac{(k/e)^k}{n^k}\geq \pg eC^2\pd^{-k} \geq \pg eC^2\pd^{-Ck^{3/2}/n^{1/2}}.
\end{equation}
This proves the desired result, with $K_1 = C\ln(eC^2)$.
\end{proof}
Given a Young diagram $\lambda$, we denote its $i$-th hook (as a set) by $\lambda^{i}$, and given $j\geq 1$, we set $\lambda^{\leq j} = \cup_{i\leq j}\lambda^i$ and $\lambda^{> j} = \cup_{i> j}\lambda^i$: these sets are the union of the first $j$ hooks, and its complement, a \textit{deep center} of the diagram.

We will also need the following intermediate estimate, which is interesting on its own. The marked diagrams in the next proof are similar to the hook functions from \cite{NovelliPakStoyanovskii1997DirectBijectiveProofHookLengthFormula}.

\begin{proposition}\label{prop: marked injection Naruse}
    Let $\mu\subseteq\lambda$ be two Young diagrams. Let $M \geq \max(\mu_1, \mu'_1)$ be an integer and set $W = |\lambda^{>M}|$. Assume that $[M^M]\subseteq \lambda$ and $W\geq |\mu|$. Then
    \begin{equation}
        S(\lambda, \mu) \geq d_\mu \binom{W}{|\mu|}.
    \end{equation}
\end{proposition}
\begin{proof}
    Let $k= |\mu|$ and $\delta = \delta(\lambda)$. For $0\leq i \leq \delta - M$, we set $W_i = W-|\lambda^{>M+i}|$: this is the number of boxes on the first $i$ hooks of $\lambda^{>M}$. Let us prove that a set of size $d_\mu \binom{W}{k}$ injects into a set of size $S(\lambda, \mu)$. 

On the one hand, consider the set $\ST_W(\mu)$ of standard tableaux of $\mu$ filled with $k$ numbers from $1$ to $W$ (instead of all numbers from $1$ to $k$). Its cardinality is $d_\mu \binom{W}{k}$.

On the other hand, consider the set $\cE_{\textup{marked}}(\lambda, \mu)$ of marked excited diagrams of $\mu$ in $\lambda$, that is, the set $\cE(\lambda, \mu)$ of excited diagrams of $\mu$ in $\lambda$ where each box $u$ is labelled with an integer from 1 to its hook length $H(\lambda, u)$. Its cardinality is $S(\lambda, \mu)$.

Now, consider the map which sends a tableau $T \in \ST_W(\mu)$ to a marked excited diagram, by applying the following to each box $u\in T$. Denote the label of $u$ by $\ell$. Let $i$ be the unique index such that $W_{i-1}<\ell \leq W_i$. Then move $u$ by $i-1$ steps along its diagonal (that is, add $(i-1, i-1)$ to the coordinates) and replace its label by $\ell - W_{i-1}$. We illustrate this process in Figure~\ref{fig: marked Naruse injection}.

Let us prove that it is well defined. Let $T\in \ST_W(\mu)$.
Write $\delta' = \delta-M$. Then each box of $T$ is moved in diagonal by at most $\delta'-1$ by construction. Since it starts in the square $[M^M]$ it ends in the square $[(M+\delta'-1)^{M+\delta'-1}] = [(\delta-1)^{\delta-1}]$, and hence in $\lambda$. Now, due to the standard constraints on the labels in $T$, the boxes higher in the same column or farther right in the same row of a given box $u$ move at least as much as $u$, that is, the displacement of boxes is weakly increasing along rows and columns. This implies that the resulting set of boxes is an excited diagram of $\mu$. Finally, consider a box $u\in T$ with label $\ell$, let $i$ be the unique index such that $W_{i-1}<\ell \leq W_i$. Let $v=u+(i-1, i-1)$ and $z = (M+i, M+i)$, coordinatewise. Then, using that $W_{i-1}<\ell \leq W_i$ and the monotonicity of the hook lengths, we have $1\leq \ell - W_{i-1} \leq W_i-W_{i-1} = H(\lambda, z)\leq H(\lambda, v)$. In particular the new label $\ell' = \ell - W_{i-1}$ is between $1$ and $H(\lambda, v)$. Combining all this shows that the image of $T$ is indeed in $\cE_{\textup{marked}}(\lambda, \mu)$.

Now, it is easy to recover the original tableau (on the left in Figure~\ref{fig: marked Naruse injection}) from the marked diagram (on the right in Figure~\ref{fig: marked Naruse injection}), by pushing the boxes methodically along their diagonal, as close as possible to the corner, starting from the boxes on $\lambda^1$, then $\lambda^2$, and so on, and adding $W_i$ to the label of the box if it moved back by $(-i,-i)$ coordinatewise. This shows that the process is injective. 

Hence the process that we defined is an injection from $\ST_W(\mu)$ to  $\cE_{\textup{marked}}(\lambda, \mu)$, which implies that $d_\mu \binom{W}{k} \leq S(\lambda, \mu)$, as desired.
\end{proof}

\begin{figure}[!ht]
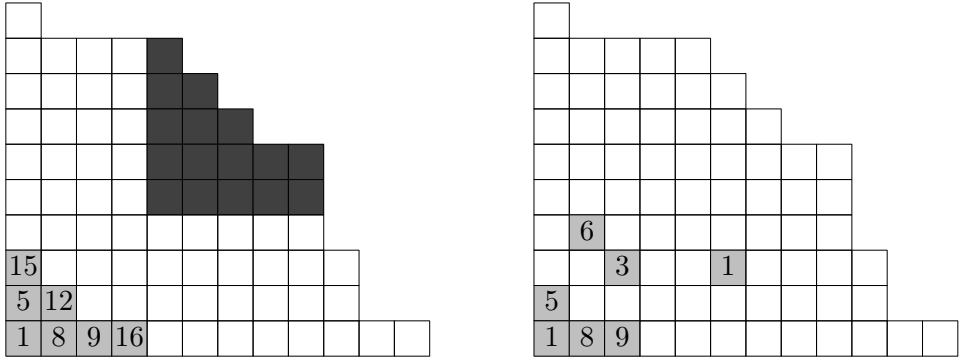

    \centering
{
\ytableausetup{boxsize= 1.2em}
\begin{ytableau}
    \none &\\
    \none &&&&&*(darkgray)  \\
    \none &&&&&*(darkgray)&*(darkgray) \\
    \none &&&&&*(darkgray)&*(darkgray)&*(darkgray) \\
    \none &&&&&*(darkgray)&*(darkgray)&*(darkgray)&*(darkgray)&*(darkgray)\\
    \none &&&&&*(darkgray)&*(darkgray)&*(darkgray)&*(darkgray)&*(darkgray) \\
    \none &&&&&&&&& \\
    \none &*(lightgray)15&&&&&&&&&\\
    \none &*(lightgray)5&*(lightgray)12&&&&&&&& \\
    \none &*(lightgray)1&*(lightgray)8&*(lightgray)9&*(lightgray)16&&&&&&&&\\
\end{ytableau}
\quad \quad \begin{ytableau}
    \none &\\
    \none &&&&&  \\
    \none &&&&&& \\
    \none &&&&&&& \\
    \none &&&&&&&&&\\
    \none &&&&&&&&& \\
    \none &&*(lightgray)6&&&&&&& \\
    \none &&&*(lightgray)3&&&*(lightgray)1&&&&\\
    \none &*(lightgray)5&&&&&&&&& \\
    \none &*(lightgray)1&*(lightgray)8&*(lightgray)9&&&&&&&&&\\
\end{ytableau}
}
    \caption{On the left: in light gray the diagram $\mu$, $k=7$, $M = 4$, in dark gray the diagram $\lambda^{>M}$, $ W_0 = 0, W_1 = 9, W_2 = 15, W_3 = W = 16$, the boxes of $\mu$ are filled in a standard way with numbers from $1$ to $W = 16$. On the right, the corresponding marked excited diagram after applying the injection. The boxes with initial label $\ell \in \ag1,5,8,9\ad$ satisfy $W_0 < \ell \leq W_1$, so they do not move nor change their label. The boxes with initial label $\ell \in \ag 12, 15\ad$ satisfy $W_1<\ell \leq W_2$, so they move by 1 step in diagonal and their label is decreased by $W_1 =9$. The box with initial label $\ell = 16$ satisfies $W_2<\ell \leq W_3$, so it moves by 2 steps in diagonal and its label is decreased by $W_2 =15$.}
    \label{fig: marked Naruse injection}
\end{figure}

\begin{proof}[Proof of Theorem~\ref{thm: skew diagrams upper bound Dousse Feray conjecture} (lower bound part)]
  Let $\mu \subseteq \lambda$ be two $C$-balanced Young diagrams. Without loss of generality we may assume that $C\geq 2$. Denote $k = |\mu|$ and $n=|\lambda|$. If $k\geq C^{-8}n$, then the result holds by Lemma~\ref{lem: lower bound Dousse Feray k order n} (applied with $C^4$ instead of $C$).
    
    From now on we assume that $k\leq C^{-8}n$. 
Set $M = \lfloor C\sqrt{k} \rfloor$. The $C$-balance assumptions together with the condition $k\leq C^{-8}n$ then ensure that $\mu \subseteq [M^M] \subseteq \lambda$. We also denote $S = |\lambda^{\leq M}|$ and $W = |\lambda^{>M}|$.

We have $S \leq s(\lambda) M \leq C\sqrt{n} C\sqrt{k} = C^2\sqrt{kn}$. Therefore $W-k\geq n/2$, and we also have $W\geq k$. We deduce that
\begin{equation}
    \frac{n^{\downarrow k}}{W^{\downarrow k}} = \prod_{i=0}^{k-1} \pg 1+\frac{S}{W-i}\pd \leq \pg e^{S/(n/2)} \pd^k \leq e^{2C^2k^{3/2}/n^{1/2}}.
\end{equation}
It follows that 
\begin{equation}
    A_{\lambda, \mu} = k! \frac{d_{\lambda\backslash\mu}}{d_\mu d_\lambda} = \frac{W^{\downarrow k}}{n^{\downarrow k}} \frac{S(\lambda, \mu)}{d_\mu \binom{W}{k}} \geq e^{-2C^2k^{3/2}/n^{1/2}}\frac{S(\lambda, \mu)}{d_\mu \binom{W}{k}}.
\end{equation}
This concludes the proof, since $d_\mu \binom{W}{k} \leq S(\lambda, \mu)$ by Proposition~\ref{prop: marked injection Naruse}.
\end{proof}

\section*{Acknowledgements}
We thank Valentin Féray for interesting related conversations.

This work was supported by the ANR project LOUCCOUM (ANR-24-CE40-7809).

\section*{AI-use statement}
The author (LT) used ChatGPT 6 Astra (C6A) to discuss proof ideas, literature review, and proofreading. Prior work on the topic and initial discussions with C6A convinced LT that usual peeling techniques would not work, even for square diagrams. The idea of considering river ribbon tableaux was then found by LT. Once this idea and the intuition behind it were shared with C6A, it completed the proof for square diagrams in one go. After a few iterations, the result was generalized to balanced diagrams and then to general shapes.

The paper was fully typed by LT. All proofs were rewritten and some were further simplified. LT takes full responsibility for the contents of the paper.

{\footnotesize
\bibliographystyle{alpha}
\bibliography{bibliographieLucas}
}
\end{document}